\documentclass[10pt,twoside]{siamart1116}

\usepackage[T1]{fontenc}
\usepackage{lmodern}
\usepackage{microtype}
\usepackage[english]{babel}
\usepackage{amsfonts,amssymb,mathtools}
\usepackage{booktabs,longtable,array,tabularx}
\usepackage{float}
\usepackage{enumitem}

\renewcommand{\theequation}{\thesection.\arabic{equation}}
\renewtheorem{theorem}{Theorem}[section]

\newcommand{\Names}{S. Zeng}
\newcommand{\Title}{The General Subgroup Permanental-Dominance Conjecture in Order Four}
\newcommand{\ShortTitle}{Permanental dominance in order four}

\newcommand{\CC}{\mathbb C}
\newcommand{\RR}{\mathbb R}

\newcommand{\Sn}{\mathfrak S}
\newcommand{\per}{\operatorname{per}}

\newcommand{\diag}{\operatorname{diag}}
\newcommand{\Rea}{\operatorname{Re}}
\newcommand{\Ima}{\operatorname{Im}}
\newcommand{\adj}{\operatorname{adj}}

\newcommand{\ip}[2]{\left\langle #1,#2\right\rangle}
\newcommand{\norm}[1]{\left\lVert #1\right\rVert}

\newcommand{\mat}[1]{\begin{pmatrix}#1\end{pmatrix}}
\newcommand{\psd}{\succeq 0}
\newcommand{\pd}{\succ 0}

\hypersetup{
  colorlinks=true,
  linkcolor=blue!55!black,
  citecolor=blue!55!black,
  urlcolor=blue!55!black,
  pdftitle={The General Subgroup Permanental-Dominance Conjecture in Order Four},
  pdfauthor={Siwei Zeng},
  pdfsubject={The complete order-four case of general subgroup permanental dominance},
  pdfkeywords={permanent, generalized matrix function, immanant,
    positive-semidefinite matrix, character theory, Gram matrix}
}

\begin{document}

\bibliographystyle{siamplain}
\setcounter{page}{1}
\thispagestyle{empty}

\title{\Title}

\author{Siwei Zeng\thanks{Independent Researcher, Beijing, China
(\href{mailto:endlesslethe@gmail.com}{endlesslethe@gmail.com}).
ORCID: \href{https://orcid.org/0009-0008-8914-7143}{0009-0008-8914-7143}.}}

\markboth{\Names}{\ShortTitle}
\maketitle

\begin{abstract}
The general subgroup permanental-dominance conjecture was previously known
only through matrix order three.  This paper proves its complete order-four
case: for every subgroup $H\leq\Sn_4$, every irreducible complex character
$\chi$ of $H$, and every $4\times4$ Hermitian positive-semidefinite matrix
$A$, it establishes $d_\chi^H(A)/\chi(1)\leq\per A$.  Unlike the usual
immanant specialization, the result covers all thirty-seven
irreducible-character cases arising from the eleven conjugacy classes of
subgroups of $\Sn_4$.  Thirty-five cases follow from general principal-minor,
moment, and block-contraction inequalities.  The two non-real $A_4$
characters are reduced to polynomial nonnegativity on the cone of
$3\times3$ positive-semidefinite Gram matrices and are resolved by a
rank-one sum-of-squares identity, an exact positive-definite interior
certificate, rational Gram certificates, and closure of the Gram cone.
\end{abstract}

\begin{keywords}
permanent, generalized matrix function, immanant, positive-semidefinite
matrix, character theory, Gram matrix
\end{keywords}

\begin{AMS}
15A15, 15B57, 20C15
\end{AMS}

\section{Introduction and main theorem}\label{sec:introduction}

The general subgroup permanental-dominance conjecture asks whether every
normalized generalized matrix function associated with an irreducible
character of every subgroup of $\Sn_n$ is bounded above by the permanent on
the Hermitian positive-semidefinite cone.  Before the present work, this
form of the conjecture was known only through matrix order three
\cite{tabata2015,wanless2022}.  The main result below establishes the complete
order-four case.  It is not merely the five-character $\Sn_4$ immanant
specialization: the eleven conjugacy classes of subgroups of $\Sn_4$ produce
thirty-seven irreducible-character cases, including two non-real $A_4$
characters that require a new Gram-cone argument.

The determinant and permanent are the two best-known members of the family
of generalized matrix functions.  Schur's classical theorem places the
determinant below every normalized irreducible immanant on the Hermitian
positive-semidefinite cone \cite{schur1918}.  Marcus proved the permanent
analogue of Hadamard's inequality \cite{marcus1963}, and Lieb strengthened
that direction through a block inequality and conjectured that the permanent
lies above every normalized immanant \cite{lieb1966}.  The formulation became
known as the permanental-dominance conjecture through the work of Merris
\cite{merris1987}.  In its subgroup form, for an irreducible character $\chi$
of $H\leq\Sn_n$, the conjectured inequality is
\[
  \frac{1}{\chi(1)}
  \sum_{\sigma\in H}\chi(\sigma)\prod_{i=1}^n A_{i,\sigma(i)}
  \leq \per A
  \qquad (A\psd).
\]
Taking $H=\Sn_n$ recovers the usual irreducible-immanant statement.
This all-subgroup conjecture is distinct from the stronger
permanent-on-top conjecture, which asserts that the permanent is the largest
eigenvalue of the Schur power matrix and is known to be false in general
\cite{wanless2022}.

The conjecture has generated a substantial literature on dominance
orderings and sufficient analytic mechanisms.  Representative advances
include identities and inequalities for generalized matrix functions
\cite{merriswatkins1985}, the ordering of hook immanants
\cite{heyfron1988}, row-appending and tensor methods
\cite{pate1991,pate1998,pate1999}, and the spectral approach of Soules
\cite{soules1994}.  Broader accounts and historical surveys are given by
Cheon and Wanless \cite{cheonwanless2005}, Zhang \cite{zhang2016}, and
Wanless \cite{wanless2022}.  Rodtes obtained further identities and explicit
sufficient criteria for classes of positive-semidefinite matrices
\cite{rodtes2024}.  More recent work relates immanant inequalities to
weight-space methods \cite{jingliu2025}, while the immanant form of
permanental dominance has acquired an operational interpretation through
generalized bosonic bunching probabilities \cite{gellerknill2026}.

Most of the literature concerns the immanant specialization $H=\Sn_n$.
That specialization is known in several broad families and, in particular,
through order thirteen; these results do not by themselves cover characters
of proper subgroups \cite{pate1998,zhang2016,wanless2022}.  For matrix order
$n=3$, Tabata
completed the description of the possible positions of normalized
generalized matrix functions, including the remaining proper-subgroup cases
\cite{tabata2015}.  In order four, Barrett, Hall, and Loewy characterized the
cone of class-function inequalities and hence the inequalities among linear
combinations of $\Sn_4$-immanants \cite{barretthallloewy1999}.  James also
exhibited a positive-semidefinite matrix giving equality for one of the
non-real $A_4$ characters
\cite{james1992,cheonwanless2005,wanless2022}.

These order-four results do not settle the full subgroup problem.  A
character of a proper subgroup, extended by zero to $\Sn_4$, need not be a
class function on $\Sn_4$; in particular, odd conjugation interchanges the
two non-real characters of $A_4$.  The distinction produces a substantially
larger finite problem: $\Sn_4$ has five irreducible characters, whereas all
subgroups of $\Sn_4$ contribute thirty-seven irreducible-character cases.
The theorem below closes this order-four gap.

For $n\geq1$, $H\leq\Sn_n$, a character $\chi:H\to\CC$, and
$A=(A_{ij})\in M_n(\CC)$, write
\begin{equation}\label{eq:def-dchi}
 d_\chi^H(A):=\sum_{\sigma\in H}\chi(\sigma)
   \prod_{i=1}^n A_{i,\sigma(i)},
 \qquad
 \per A:=\sum_{\sigma\in\Sn_n}\prod_{i=1}^n A_{i,\sigma(i)}.
\end{equation}
The normalization in the conjecture is $d_\chi^H/\chi(1)$; no complex
conjugate is placed on $\chi(\sigma)$ in~\eqref{eq:def-dchi}.

\begin{theorem}[Permanental dominance in order four]\label{thm:main}
Let $H\leq\Sn_4$, let $V$ be an irreducible finite-dimensional complex
representation of $H$, and let $\chi$ be its character.  For every
$A\in M_4(\CC)$ with $A\psd$,
\begin{equation}\label{eq:main}
  \frac{d_\chi^H(A)}{\dim V}\leq \per A.
\end{equation}
\end{theorem}

Theorem~\ref{thm:main} concerns matrix order four only; it makes no assertion
in arbitrary order.

The proof has three deliberately separated layers.  Its conceptual step is
the Gram-cone reduction for the two non-real $A_4$ characters.  Its
completeness layer classifies all thirty-seven irreducible-character cases
and resolves the other thirty-five by general inequalities.  Its
verification layer consists of exact polynomial certificates, supplemented
by a Lean formalization.  The latter two layers certify the scope and the
algebra without obscuring the geometric mechanism.

The organizing idea is to write the exceptional gap as
\[
 c_0(u,v,w)+q^*K(u,v,w)q.
\]
The block positivity constraint says that $B(u,v,w)-qq^*$ is a Gram matrix.
After dividing out the nonzero coordinates of $q$, this residual becomes the
Gram matrix $R$ of three vectors, and the exceptional gap factors as
\[
 |q_1q_2q_3|^2\mathcal P(R).
\]
The problem is therefore moved from a character sum on four indices to the
nonnegativity of one explicit polynomial on the cone of $3\times3$ Gram
matrices.

Section~\ref{sec:statement-reductions} performs the analytic reductions and
proves that the finite ledger is complete.  Section~\ref{sec:elementary-cases}
settles the thirty-five cases governed by general inequalities.
Section~\ref{sec:a4-geometry} converts the exceptional character gap to the
Gram polynomial, and Section~\ref{sec:Ppositive} proves its nonnegativity and
completes Theorem~\ref{thm:main}.  Section~\ref{sec:matrix-consequence}
records an adjugate-matrix consequence.  \ref{app:character-data} contains
the complete case data, \ref{app:sos} contains the exact coefficient
certificates, and \ref{app:formal-verification} maps
the paper proof to its Lean counterparts.

\section{Reductions and the finite case ledger}\label{sec:statement-reductions}

\subsection{Reality and reduction to correlation matrices}

\begin{lemma}[Reality]\label{lem:reality}
If $A=A^*$ and $\chi$ is the character of a finite-dimensional complex
representation of $H$, then $d_\chi^H(A)$ and $\per A$ are real.
\end{lemma}

\begin{proof}
Every finite-group representation may be made unitary, and hence
$\overline{\chi(\sigma)}=\chi(\sigma^{-1})$.  Hermitian symmetry gives
\[
 \overline{\prod_i A_{i,\sigma(i)}}
 =\prod_i A_{\sigma(i),i}
 =\prod_j A_{j,\sigma^{-1}(j)}.
\]
Conjugating the sum and replacing $\sigma$ by $\sigma^{-1}$ proves the first
claim.  The second is the special case $H=\Sn_n$, $\chi=1$.
\end{proof}

The normalization step is elementary but important because it removes all
diagonal parameters without an invertibility hypothesis.

\begin{lemma}[Diagonal congruence]\label{lem:scaling}
Let $s_1,\dots,s_n\in\RR$, and let
\[
 D=\diag(s_1,\dots,s_n).
\]
Then
\[
 d_\chi^H(DAD)=\left(\prod_i s_i\right)^2d_\chi^H(A),
 \qquad
 \per(DAD)=\left(\prod_i s_i\right)^2\per A.
\]
\end{lemma}

\begin{proof}
For every permutation $\sigma$,
\[
 \prod_i(DAD)_{i,\sigma(i)}
 =\prod_i(s_iA_{i,\sigma(i)}s_{\sigma(i)})
 =\left(\prod_i s_i\right)^2\prod_iA_{i,\sigma(i)}.
\]
Summation proves both identities.
\end{proof}

\begin{lemma}[Correlation reduction]\label{lem:correlation-reduction}
It is enough to prove Theorem~\ref{thm:main} for positive-semidefinite
matrices whose diagonal entries are all one.
\end{lemma}

\begin{proof}
Suppose first that $A_{ii}>0$ for all $i$ and put
$s_i=\sqrt{A_{ii}}$.  Since a Hermitian matrix has real diagonal, the matrix
\[
 C=\diag(s_i^{-1})A\diag(s_i^{-1})
\]
is positive semidefinite and has diagonal one.  Moreover,
\[
 A=\diag(s_i)C\diag(s_i).
\]
Thus Lemma~\ref{lem:scaling} transfers the desired inequality from $C$ to
$A$.

If $A_{ii}=0$ for some $i$, positivity of the $2\times2$ principal minors
implies $A_{ij}=A_{ji}=0$ for every $j$.  Each permutation monomial then
contains a zero entry in row $i$, so both sides of~\eqref{eq:main} vanish.
\end{proof}

We therefore write a general $4\times4$ correlation matrix as
\begin{equation}\label{eq:corr-six}
 A(a,b,c,d,e,f)=
 \mat{
 1&a&b&c\\
 \bar a&1&d&e\\
 \bar b&\bar d&1&f\\
 \bar c&\bar e&\bar f&1
 }\psd.
\end{equation}

\subsection{Subgroups and characters}\label{sec:finite-reduction}

Table~\ref{tab:s4-subgroups} lists the eleven conjugacy types of subgroups of
$\Sn_4$.
Here the superscripts ${\rm t},{\rm d},{\rm n}$ denote transposition,
double transposition, and normal, respectively; in the dihedral row,
$r=(1234)$ and $s=(13)$.
The case column in Table~\ref{tab:s4-subgroups} sums to
\[
 1+2+2+3+4+4+4+3+5+4+5=37.
\]
Proposition~\ref{prop:elementary-cases} resolves thirty-five cases, and
Proposition~\ref{prop:a4-characters} resolves the two non-real $A_4$ cases.

\begin{lemma}[Subgroup classification]\label{lem:subgroup-classification}
Every subgroup of $\Sn_4$ is conjugate to exactly one of the eleven types in
Table~\ref{tab:s4-subgroups}.
\end{lemma}

\begin{proof}
An intransitive subgroup preserves its orbit partition.  For partition
$2+1+1$ it is contained in $S_2$ and gives $1$ or $C_2^{\rm t}$.  For
$2+2$ it is contained in $S_2\times S_2$ and gives, up to conjugacy,
$1,C_2^{\rm t},C_2^{\rm d}$, or $V_4^{\rm d}$.  For $3+1$ it is contained
in the displayed point stabilizer $S_3$ and gives
$1,C_2^{\rm t},C_3$, or $S_3$.  This accounts for every intransitive type.

If the action is transitive, orbit--stabilizer shows that $4$ divides
$|H|$, while Lagrange's theorem gives $|H|\mid24$.  A transitive group of
order four is regular and is either $C_4$ or the normal
double-transposition group $V_4^{\rm n}$.  A group of order eight is a Sylow
$2$-subgroup, hence is conjugate to
$\langle(1234),(13)\rangle\cong D_8$.  The cases of order twelve and
twenty-four are $A_4$ and $S_4$, respectively.  The listed types are
distinguished by order, orbit structure, and, for the two order-two and two
Klein-four types, by cycle type.
\end{proof}

\begin{table}[htbp]
\caption{Finite ledger of the irreducible-character cases for subgroups of $\Sn_4$.}
\label{tab:s4-subgroups}
\centering
\footnotesize
\begin{tabularx}{\textwidth}{@{}l>{\raggedright\arraybackslash}Xlrr>{\raggedright\arraybackslash}X@{}}
\toprule
type & representative & group & $|H|$ & cases & resolution\\
\midrule
$1$ & $\{1\}$ & $1$ & 1&1& principal permanent\\
$C_2^{\rm t}$ & $\langle(12)\rangle$ & $C_2$ &2&2& principal permanent\\
$C_2^{\rm d}$ & $\langle(12)(34)\rangle$ & $C_2$ &2&2& block contraction\\
$C_3$ & $\langle(123)\rangle$ & $C_3$ &3&3& scalar triangle bounds\\
$C_4$ & $\langle(1234)\rangle$ & $C_4$ &4&4& scalar and block bounds\\
$V_4^{\rm n}$ & $\{1,(12)(34),$\linebreak
$(13)(24),(14)(23)\}$ & $C_2^2$ &4&4& moment and scalar bounds\\
$V_4^{\rm d}$ & $\langle(12),(34)\rangle$ & $C_2^2$ &4&4& block contraction\\
$S_3$ & $\operatorname{Stab}(4)$ & $S_3$ &6&3& principal and scalar bounds\\
$D_8$ & $\langle r,s\rangle$ & $D_8$ &8&5& block contractions\\
$A_4$ & $A_4$ & $A_4$ &12&4& general bounds (2); Gram cone (2)\\
$S_4$ & $S_4$ & $S_4$ &24&5& moment and scalar bounds\\
\bottomrule
\end{tabularx}
\end{table}

The character data used below are recorded in
\ref{app:character-data}.  Orthogonality of the character-table rows
and the identity $\sum_\chi\chi(1)^2=|H|$ prove that no irreducible character
is missing.  For the two exceptional $A_4$ characters we fix
\begin{equation}\label{eq:omega}
 \omega=-\frac12-\frac{\sqrt3}{2}i,
 \qquad \omega^2+\omega+1=0.
\end{equation}

\subsection{Conjugation and finite reduction}

Simultaneous relabeling transports both the subgroup and its character data.

\begin{lemma}[Transport]\label{lem:transport}
Let $g\in\Sn_4$, $K=g^{-1}Hg$, and let
$\chi^g(g^{-1}\sigma g)=\chi(\sigma)$.  If
$A^g_{ij}=A_{g(i),g(j)}$, then
\[
 d_{\chi^g}^{K}(A^g)=d_\chi^H(A),\qquad \per A^g=\per A,
 \qquad A\psd\Longrightarrow A^g\psd.
\]
\end{lemma}

\begin{proof}
Simultaneous permutation of rows and columns preserves positivity and the
permanent.  For a permutation monomial, changing $i$ to $g(i)$ gives
\[
 \prod_i A^g_{i,(g^{-1}\sigma g)(i)}
 =\prod_j A_{j,\sigma(j)}.
\]
Summing proves the assertion.
\end{proof}

Thus it suffices to prove dominance in the thirty-seven irreducible-character
cases associated with the eleven fixed representatives.

\section{General inequalities and the thirty-five cases}\label{sec:elementary-cases}

\subsection{Four cycle aggregates}

For~\eqref{eq:corr-six} define
\begin{align}
 T={}&|a|^2+|b|^2+|c|^2+|d|^2+|e|^2+|f|^2,\label{eq:T}\\
 D={}&|a|^2|f|^2+|b|^2|e|^2+|c|^2|d|^2,\label{eq:D}\\
 C={}&2\Rea(a d\bar b+a e\bar c+b f\bar c+d f\bar e),\label{eq:C}\\
 F={}&2\Rea(a d f\bar c+a e\bar f\bar b+b\bar d e\bar c).
 \label{eq:F}
\end{align}
Enumerating the cycle types of $\Sn_4$ gives
\begin{equation}\label{eq:per-aggregate}
 \per A=1+T+D+C+F.
\end{equation}
The five normalized $S_4$ generalized matrix functions, in the order of
the characters in~\eqref{eq:s4-table}, are
\begin{align}
 p_{[4]}&=1+T+D+C+F,\nonumber\\
 p_{[1^4]}&=1-T+D+C-F,\nonumber\\
 p_{[31]}&=1+\tfrac13T-\tfrac13D-\tfrac13F,\label{eq:s4-rows}\\
 p_{[211]}&=1-\tfrac13T-\tfrac13D+\tfrac13F,\nonumber\\
 p_{[22]}&=1+D-\tfrac12C.\nonumber
\end{align}

\begin{lemma}[Scalar correlation inequalities]\label{lem:scalar-toolbox}
If $A$ is the correlation matrix in~\eqref{eq:corr-six}, then
\begin{gather}
 0\leq T,\quad 0\leq D,\quad
 D\leq\frac{T^2}{4},\quad
 -2D\leq F\leq2D\leq T,
 \label{eq:scalar-basic}\\
 \frac{T^2}{2}\leq T+\frac32C,
 \qquad 1+D\leq\per A.\label{eq:moment-bound}
\end{gather}
In addition, if
\begin{equation}\label{eq:triangle-data}
 U=|a|^2+|b|^2+|d|^2,\qquad z=ad\bar b,
\end{equation}
then
\begin{equation}\label{eq:triangle-character-bounds}
 U+3\Rea z-\sqrt3\Ima z\geq0,
 \qquad U+3\Rea z+\sqrt3\Ima z\geq0.
\end{equation}
Every $3\times3$ principal permanent is at most $\per A$.
\end{lemma}

\begin{proof}
The inequalities $|a|^2,\ldots,|f|^2\leq1$ follow from the $2\times2$
principal minors.  The first four estimates in~\eqref{eq:scalar-basic} are
then AM--GM and
$2|\Rea(xy)|\leq|x|^2+|y|^2$, paired over opposite edges.

For the moment inequality, let $\lambda_1,\dots,\lambda_4\geq0$ be the
eigenvalues of $A$.  Direct trace expansion gives
\[
 \sum\lambda_i=4,\quad
 \sum\lambda_i^2=4+2T,\quad
 \sum\lambda_i^3=4+6T+3C.
\]
The exact identity
\[
 \left(\sum_i\lambda_i\right)\left(\sum_i\lambda_i^3\right)
 -\left(\sum_i\lambda_i^2\right)^2
 =\sum_{i<j}\lambda_i\lambda_j(\lambda_i-\lambda_j)^2
\]
implies the first inequality in~\eqref{eq:moment-bound}.  It also gives
$C\geq(T^2-2T)/3$.  To prove the second inequality,
subtract $1+D$ from~\eqref{eq:per-aggregate}.  If
$T\leq2$, then
\[
 T+C+F
 \geq T+\frac{T^2-2T}{3}-2D
 \geq \frac{T(2-T)}6\geq0.
\]
If $T\geq2$, the moment bound gives $C\geq0$, while
$2D\leq T$ and $F\geq-2D$ give the same
conclusion.  In particular, $1\leq\per A$.

We next prove~\eqref{eq:triangle-character-bounds}.  Put
$u=|a|^2$, $v=|b|^2$, $w=|d|^2$, $S=u+v+w$, and write $z=x+iy$.
Positivity of the principal $3\times3$ block gives
\[
 0\leq u,v,w\leq1,\qquad x^2+y^2=uvw,
 \qquad 1-S+2x\geq0.
\]
Also $27uvw\leq S^3$ by AM--GM.  If $S\leq9/4$, Cauchy--Schwarz gives
\[
 (3x\pm\sqrt3y)^2\leq12(x^2+y^2)
 \leq\frac49S^3\leq S^2,
\]
which proves both inequalities.  If $S\geq9/4$, then
$x\geq(S-1)/2$ and $|y|\leq1$, so
\[
 S+3x\pm\sqrt3y
 \geq\frac{5S-3}{2}-\sqrt3>0.
\]

It remains to prove the principal-permanent contraction.  Choose Gram
vectors $x_1,\ldots,x_4$ for $A$, write
$x\odot y=x\otimes y+y\otimes x$, and set
\[
 y_1=x_2\odot x_3,\qquad y_2=x_1\odot x_3,\qquad
 y_3=x_1\odot x_2,\qquad
 v=\bar c y_1+\bar e y_2+\bar f y_3.
\]
The identity
\[
 \ip{x\odot y}{z\odot w}
 =2\bigl(\ip{x}{z}\ip{y}{w}+\ip{x}{w}\ip{y}{z}\bigr)
\]
and direct collection of the cycle types give
\[
 \norm v^2=2\bigl(\per A-per A[\{1,2,3\}]\bigr).
\]
Hence the displayed principal permanent is at most $\per A$; simultaneous
relabeling proves the other three cases.
\end{proof}

We also need two symmetric $2+2$ contractions.  Put
\begin{align}
 R_+&:=1+|b|^2+|e|^2+D
       +2\Rea(a d f\bar c),\label{eq:Rplus}\\
 R_\wedge&:=(1-|b|^2)(1-|e|^2)+|af-c\bar d|^2.\label{eq:Rwedge}
\end{align}

\begin{lemma}[Block contractions]\label{lem:block-contractions}
If $A$ is the correlation matrix in~\eqref{eq:corr-six}, then
\begin{equation}\label{eq:block-bounds}
 R_+\leq\per A,\qquad R_\wedge\leq\per A,\qquad
 1+|a|^2|f|^2+|c|^2|d|^2\leq\per A.
\end{equation}
\end{lemma}

\begin{proof}
First note the exact identity
\begin{equation}\label{eq:Rplus-square}
 R_+=(1+|b|^2)(1+|e|^2)+|af+c\bar d|^2.
\end{equation}
Put $r=(1-|b|^2)^{1/2}$ and
\[
 y_0=a+\bar d b,\qquad y_1=\bar d r,\qquad
 g_0=c+fb,\quad g_1=fr.
\]
If $Y=|y_0|^2+|y_1|^2$, $Z=|g_0|^2+|g_1|^2$, and
$h=\bar y_0g_0+\bar y_1g_1$, expanding the definitions gives the certificate
\begin{equation}\label{eq:Rplus-gap}
 \per A-R_+=Y+Z+2\Rea(\bar e h).
\end{equation}
Cauchy--Schwarz and $|e|\leq1$ imply
$|\Rea(\bar e h)|^2\leq YZ$.  Since
$Y+Z\geq2\sqrt{YZ}\geq2|\Rea(\bar e h)|$, the right-hand side
of~\eqref{eq:Rplus-gap} is nonnegative.  This proves the first contraction.

For the exterior contraction, take Gram vectors $x_1,\ldots,x_4$ for $A$.
Exterior Cauchy--Schwarz applied to $x_1\wedge x_3$ and
$x_2\wedge x_4$ gives
\[
 |af-c\bar d|^2\leq(1-|b|^2)(1-|e|^2),
\]
where the inner product of the two exterior tensors is $af-c\bar d$.

We need two consequences of the first contraction.  Relabeling its split
twice and using~\eqref{eq:Rplus-square} yields
$(1+|a|^2)(1+|f|^2)\leq\per A$ and
$(1+|c|^2)(1+|d|^2)\leq\per A$.  Put
\[
 x=|a|^2|f|^2,\qquad y=|c|^2|d|^2,\qquad P=\per A.
\]
If $x\leq y$, then $x\leq|c|^2$ and
$1+x+y\leq(1+|c|^2)(1+|d|^2)\leq P$; if $y\leq x$, use the other matching
bound.  This proves the third inequality in~\eqref{eq:block-bounds}.

Finally let $p=|b|^2$, $q=|e|^2$, and
$s=|af-c\bar d|$.  The exterior bound gives
$s^2\leq(1-p)(1-q)$, hence $0\leq s\leq1$.  The triangle inequality gives
$s\leq\sqrt{x}+\sqrt{y}$.  The two matching bounds and AM--GM imply
\[
 (1+\sqrt{x})^2\leq P,\qquad (1+\sqrt{y})^2\leq P.
\]
Consequently $(1+s/2)^2\leq P$.  Moreover
\[
 (1-p)(1-q)+s^2\leq(1+s/2)^2,
\]
because $p+q-pq\geq0$ and $s-3s^2/4\geq0$.  Therefore
$R_\wedge\leq P$, completing the proof.
\end{proof}

\subsection{Character cases covered by the general inequalities}

The following proposition records the finite analytic part of the proof.
\ref{app:character-data} records the normalized functions used
here.  The only cases not covered by the present proposition are the two
non-real $A_4$ characters.

\begin{proposition}[The thirty-five cases resolved by general inequalities]\label{prop:elementary-cases}
For every representative subgroup in Lemma~\ref{lem:subgroup-classification},
all normalized generalized matrix functions other than those associated with
$\chi_\omega$ and
$\bar\chi_\omega$ of $A_4$ are bounded above by $\per A$.
\end{proposition}

\begin{proof}
First consider the cyclic, Klein-four, and $S_3$ cases.  In the $C_3$ and
$S_3$ cases put $z=ad\bar b$.  The expressions in
Table~\ref{tab:elementary-character-expansions} follow by listing the elements
of the corresponding subgroup.  We give the inequalities explicitly.  Set
\[
 U=|a|^2+|b|^2+|d|^2,\quad z=x+iy,\quad
 P_3=1+U+2x.
\]
Lemma~\ref{lem:scalar-toolbox} gives $P_3\leq\per A$ and
$U+3x\pm\sqrt3y\geq0$.  Thus
\[
 P_3-(1-x\pm\sqrt3y)=U+3x\mp\sqrt3y\geq0,
\]
which verifies the $C_3$ cases; the value $1+2x$ is at most $P_3$ because
$U\geq0$.  For $S_3$, the trivial case is $P_3$, the sign case differs from
$P_3$ by $-2U$, and the standard case satisfies
$P_3-(1-x)=U+3x\geq0$, obtained by averaging the two triangle bounds.

The trivial and $C_2^{\rm t}$ cases are bounded by $P_3$ (or by $1$ after
choosing the negative sign).  For $C_2^{\rm d}$, the positive value is at
most $1+|a|^2|f|^2+|c|^2|d|^2\leq\per A$, while the negative value is at
most one.  For $V_4^{\rm n}$, the all-positive value is
$1+D\leq\per A$ and each other sign pattern is no larger.  For
$V_4^{\rm d}$, all four values are no larger than
$(1+|a|^2)(1+|f|^2)\leq\per A$.

For $C_4$, put
\[
 r=|b|^2|e|^2,\qquad z=adf\bar c.
\]
Its four normalized functions are listed in~\eqref{eq:c4-rows}.
Put $x=|a|^2|f|^2$ and $y=|c|^2|d|^2$.  The estimates
\[
 2|\Rea z|,\ 2|\Ima z|
 \leq x+y
\]
show that the two values containing $+r$ are at most
$1+r+x+y=1+D\leq\per A$.  The other two are at most
$1-r+x+y\leq1+x+y\leq\per A$ by~\eqref{eq:block-bounds}.

For $D_8$, write $p=|b|^2$, $q=|e|^2$.  Its five normalized functions
are listed in~\eqref{eq:d8-rows}.
The first and fourth are~\eqref{eq:block-bounds}.  By~\eqref{eq:Rplus-square},
the second differs from $R_+$ only by replacing
$+|af+c\bar d|^2$ with $-|af+c\bar d|^2$, and hence is no larger.  The third
and fifth are at most $1$, hence at most $\per A$.

For $S_4$, the sign case follows from
$T+F\geq T-2D\geq0$.  For the two remaining
nontrivial estimates, direct subtraction gives
\begin{align*}
 3(\per A-p_{[31]})
   &=2T+4D+3C+4F
     \geq T^2-4D\geq0,\\
 2(\per A-p_{[22]})
   &=2T+3C+2F
     \geq T^2-4D\geq0.
\end{align*}
Here we used~\eqref{eq:scalar-basic}--\eqref{eq:moment-bound}.  Finally,
$p_{[211]}\leq1\leq\per A$ because
$F\leq T$ and $D\geq0$.

Finally, the trivial $A_4$ case is the average of the trivial and sign $S_4$
cases, while its three-dimensional case is the average of the $[31]$ and
$[211]$ cases.  This proves the proposition.
\end{proof}

\section{The exceptional \texorpdfstring{$A_4$}{A4} case: reduction to the Gram cone}
\label{sec:a4-geometry}

\subsection{The exceptional character gap}

It remains to treat the two conjugate non-real characters of $A_4$.  Their
common geometric content is isolated below through their
permanental-dominance gap.

Let
\begin{equation}\label{eq:Buvw}
 B(u,v,w)=\mat{1&u&v\\\bar u&1&w\\\bar v&\bar w&1}.
\end{equation}
Define
\begin{equation}\label{eq:c0}
 c_0(u,v,w)=|u|^2+|v|^2+|w|^2
 +2\Rea\bigl((1-\omega)uw\bar v\bigr)
\end{equation}
and the Hermitian matrix
\begin{equation}\label{eq:K}
 K(u,v,w)=
 \mat{
 1&(1-\omega^2)u+v\bar w&uw+(1-\omega)v\\
 (1-\omega)\bar u+w\bar v&1&(1-\omega^2)w+v\bar u\\
 \bar u\bar w+(1-\omega^2)\bar v&(1-\omega)\bar w+u\bar v&1
 }.
\end{equation}

\begin{lemma}[Character expansion]\label{lem:a4-expansion}
For the correlation matrix~\eqref{eq:corr-six},
\begin{equation}\label{eq:a4-gap}
 \per A-d_{\chi_\omega}^{A_4}(A)
 =c_0(a,b,d)+\mat{\bar c&\bar e&\bar f}
 K(a,b,d)\mat{c\\e\\f}.
\end{equation}
For $\bar\chi_\omega$, the same formula holds after an odd simultaneous
relabeling of the four indices.
\end{lemma}

\begin{proof}
Choose the labeling of $C_+$ so that $\chi_\omega(C_+)=\omega$.  Listing the
twelve elements gives the expansion
\begin{align*}
d_{\chi_\omega}^{A_4}(A)
={}&1+|a|^2|f|^2+|b|^2|e|^2+|c|^2|d|^2\\
&+\omega\bigl(e\bar d\bar f+ad\bar b+bf\bar c+c\bar a\bar e\bigr)\\
&+\omega^2\bigl(df\bar e+ae\bar c+b\bar a\bar d+c\bar b\bar f\bigr).
\end{align*}
The four monomials on the last line are the conjugates of those on the
preceding line, in the inverse order, so the displayed quantity is real.
Subtracting it from~\eqref{eq:per-aggregate}, using
$\omega^2=-1-\omega$, and collecting the constant, diagonal, and off-diagonal
terms in $(c,e,f)$ yields~\eqref{eq:a4-gap}.  An odd permutation interchanges
the two three-cycle classes and therefore exchanges $\chi_\omega$ with its
conjugate.
\end{proof}

The remaining task is consequently the following four-vector inequality.

\begin{theorem}[Four-vector inequality]\label{thm:four-vector}
Suppose
\begin{equation}\label{eq:completion}
 \mathcal A=\mat{B(u,v,w)&q\\q^*&1}\psd,\qquad q\in\CC^3.
\end{equation}
Then
\begin{equation}\label{eq:four-vector}
 c_0(u,v,w)+q^*K(u,v,w)q\geq0.
\end{equation}
\end{theorem}

The proof is developed in the residual normalization below and in
Section~\ref{sec:Ppositive}.

\subsection{Residual-vector normalization}\label{sec:normalization}

There is a useful geometric meaning behind the next change of variables.
Realize the block matrix~\eqref{eq:completion} as the Gram matrix of four unit
vectors.  The vector corresponding to the last row and column determines a
common one-dimensional direction.  Subtracting the projections of the first
three vectors onto that direction leaves three residual vectors, whose Gram
matrix is $B-qq^*$.  Dividing the residual vectors by the coordinates of the
removed projections does not change positivity; it merely produces the
scale-free matrix $R$ below.  This is the step that turns a four-vector
completion problem into a problem on the cone of three-vector Gram matrices.

Assume first that $q_1q_2q_3\neq0$.  The Schur complement of the final
diagonal entry in~\eqref{eq:completion} is
$B-qq^*\psd$.  Congruence by
$\diag(\bar q_1^{-1},\bar q_2^{-1},\bar q_3^{-1})$ gives the residual Gram
matrix
\begin{equation}\label{eq:residual-gram}
 R=\mat{
 n_1&h_{12}&h_{13}\\
 \bar h_{12}&n_2&h_{23}\\
 \bar h_{13}&\bar h_{23}&n_3
 }\psd,
\end{equation}
where
\begin{equation}\label{eq:residual-coordinates}
 n_i=\frac{1-|q_i|^2}{|q_i|^2},\qquad
 h_{12}=\frac{u}{q_1\bar q_2}-1,\quad
 h_{13}=\frac{v}{q_1\bar q_3}-1,\quad
 h_{23}=\frac{w}{q_2\bar q_3}-1.
\end{equation}
Thus $R$ is the Gram matrix of three residual vectors $w_1,w_2,w_3$ with
$n_i=\norm{w_i}^2$ and $h_{ij}=\ip{w_i}{w_j}$.

We now define the normalized geometric polynomial.  Put
\begin{align}
 E={}&n_1n_2+|h_{12}|^2+n_1n_3+|h_{13}|^2+n_2n_3+|h_{23}|^2,\nonumber\\
 x_{12,13}={}&n_1h_{23}+h_{13}\bar h_{12},\nonumber\\
 x_{12,23}={}&h_{12}h_{23}+n_2h_{13},\nonumber\\
 x_{13,23}={}&n_3h_{12}+h_{13}\bar h_{23},\nonumber\\
 S={}&E+2\Rea(x_{12,13}+x_{12,23}+x_{13,23}),\label{eq:tensor-terms}\\
 S_\omega={}&E+2\Rea(\omega^2x_{12,13}+\omega x_{12,23}
                     +\omega^2x_{13,23}),\nonumber\\
 \tau={}&h_{13}\bar h_{12}\bar h_{23},\nonumber\\
 \Gamma={}&n_3|h_{12}|^2+n_2|h_{13}|^2+n_1|h_{23}|^2
              +3\Rea\tau+\sqrt3\Ima\tau.\nonumber
\end{align}
Finally set
\begin{equation}\label{eq:geometric-P}
 \mathcal P(R)=24+6\left(n_1+n_2+n_3
 +2\Rea(h_{12}+h_{13}+h_{23})\right)+2S-S_\omega+\Gamma.
\end{equation}

\begin{lemma}[Exact normalization identity]\label{lem:normalization-identity}
With~\eqref{eq:residual-coordinates},
\begin{equation}\label{eq:normalization-identity}
 c_0(u,v,w)+q^*K(u,v,w)q
 =|q_1q_2q_3|^2\mathcal P(R).
\end{equation}
\end{lemma}

\begin{proof}
Equations~\eqref{eq:residual-coordinates} give
\[
 u=q_1\bar q_2(1+h_{12}),\quad
 v=q_1\bar q_3(1+h_{13}),\quad
 w=q_2\bar q_3(1+h_{23}),
\]
and $|q_i|^2(1+n_i)=1$.  Let $Q=|q_1q_2q_3|^2$.  Dividing first by $Q$
organizes the cancellations by degree.  The two exact quotient expansions
are recorded in~\eqref{eq:c0-normalized}--\eqref{eq:K-normalized}.
They are obtained term by term;
for example,
\[
 \frac{|u|^2}{Q}=(1+n_3)|1+h_{12}|^2,\qquad
 \frac{uw\bar v}{Q}=(1+h_{12})(1+h_{23})(1+\bar h_{13}).
\]
Expanding the two quotient formulas and using
$1+\omega+\omega^2=0$ groups the degree-two terms as $2S-S_\omega$ and the
degree-three terms as $\Gamma$; the remaining terms are the constant and
linear part in~\eqref{eq:geometric-P}.  Thus their sum is
$\mathcal P(R)$, which proves~\eqref{eq:normalization-identity}.
\end{proof}

Thus Theorem~\ref{thm:four-vector} follows from the next result.

\begin{theorem}[Geometric polynomial]\label{thm:Ppositive}
For every Hermitian matrix $R$ of the form~\eqref{eq:residual-gram},
\[
 R\psd\quad\Longrightarrow\quad\mathcal P(R)\geq0.
\]
\end{theorem}

\section{Positivity on the Gram cone}\label{sec:Ppositive}

The proof uses the geometry of the Gram cone in three stages.  On the
rank-one locus the three residual vectors are collinear and the polynomial
has a direct sum-of-squares form.  In the positive-definite region, two
residual vectors are used as coordinates and the third is minimized through
an explicit two-dimensional Hermitian quadratic form.  Finally, every
singular Gram matrix is obtained as a limit of positive-definite ones.  The
rank-one calculation is also the algebraic base coefficient in the interior
minimum certificate, so it is both explanatory and logically active.

\subsection{The rank-one boundary}

Suppose first that $R$ has rank one.  Write
\[
 n_1=|x|^2,\quad n_2=|y|^2,\quad n_3=|z|^2,\qquad
 h_{12}=\bar xy,\quad h_{13}=\bar xz,\quad h_{23}=\bar yz.
\]
Define
\begin{align}
 R_0(x,y,z)={}&4+|x+y+z|^2+|xy+xz+yz|^2\nonumber\\
 &\quad-(|xy|^2+|xz|^2+|yz|^2)+|xyz|^2,\label{eq:R0}\\
 \Phi(x,y,z)={}&xy+\omega xz+\omega^2yz.
\end{align}

\begin{lemma}[Rank-one sum of squares]\label{lem:rank-one-sos}
For all $x,y,z\in\CC$,
\begin{equation}\label{eq:P-rank-one}
 \mathcal P(R)=2\bigl(3R_0(x,y,z)+|\Phi(x,y,z)|^2\bigr)\geq0.
\end{equation}
\end{lemma}

\begin{proof}
Regard $R_0$ as a quadratic in $z$.  Put
\[
 K_0=|1+\bar xy|^2,\quad
 L_0=\overline{x+y}+\overline{xy}(x+y),\quad
 C_0=4+|x+y|^2.
\]
Direct expansion gives
\[
 R_0=K_0|z|^2+2\Rea(L_0z)+C_0
\]
and the exact discriminant identity
\[
 C_0K_0-|L_0|^2=4(1+\Rea(\bar xy))^2\geq0.
\]
If $K_0>0$, completing the square proves $R_0\geq0$.  If $K_0=0$, the
discriminant identity forces $L_0=0$, and then $R_0=C_0>0$.  A second direct
expansion using $1+\omega+\omega^2=0$ gives the equality
in~\eqref{eq:P-rank-one}; its right-hand side is nonnegative.
\end{proof}

\subsection{The positive-definite region}

It remains to treat a positive-definite residual Gram matrix.  Write
\begin{equation}\label{eq:G-data}
 a=n_1>0,\quad b=n_2,\quad c=n_3,\quad
 h_{12}=g=r-is,\quad p=h_{13},\quad q=h_{23},
\end{equation}
and put
\begin{equation}\label{eq:Gram-G-delta}
 G=\mat{a&g\\\bar g&b},\qquad
 \delta=\det G=ab-r^2-s^2>0,\qquad h=\mat{p\\q}.
\end{equation}
There is a unique $\zeta\in\CC^2$ such that $h=G\zeta$.  Positivity of
$R$ tested on $(-\zeta,1)$ gives the Schur slack
\begin{equation}\label{eq:schur-slack}
 \zeta^*G\zeta\leq c.
\end{equation}

Set
\begin{align}
 \lambda={}&6+a+b+r^2+s^2+5r-\sqrt3s,\label{eq:lambda}\\
 P_0={}&24+6(a+b+2r)+ab+r^2+s^2,\label{eq:P0}\\
 \ell={}&\mat{
 6+(2-\omega)g+(2-\omega^2)b\\
 6+(2-\omega)a+(2-\omega^2)\bar g},\label{eq:ell}\\
 M={}&\mat{
 1+b&2-\omega+(1-\omega)g\\
 2-\omega^2+(1-\omega^2)\bar g&1+a},\label{eq:M}\\
 \Theta={}&\lambda I_2+MG.\label{eq:Theta-matrix}
\end{align}

\begin{lemma}[Quadratic decomposition]\label{lem:quadratic-decomposition}
With the preceding notation,
\begin{align}
 \mathcal P(R)
 ={}&P_0+2\Rea(\ell^*h)+\lambda c+\Rea(h^*Mh)\label{eq:P-quadratic}\\
 ={}&Q(\zeta)+\lambda(c-\zeta^*G\zeta),\nonumber
\end{align}
where
\begin{equation}\label{eq:Qzeta}
 Q(\zeta)=P_0+2\Rea((G\ell)^*\zeta)+\zeta^*(G\Theta)\zeta.
\end{equation}
Moreover $G\Theta$ is Hermitian.
\end{lemma}

\begin{proof}
Substitution of~\eqref{eq:G-data} in~\eqref{eq:geometric-P} gives the first
identity.  Replacing $h$ by $G\zeta$, using $G=G^*$ and
$\Theta=\lambda I+MG$, gives the second.  Finally
$G\Theta=\lambda G+GMG$ is Hermitian because $G$ and $M$ are Hermitian.
\end{proof}

\begin{lemma}[Positivity of the scalar part]\label{lem:lambda-positive}
Under $a,b\geq0$ and $r^2+s^2\leq ab$, one has $\lambda>0$.
\end{lemma}

\begin{proof}
Let $\rho=\sqrt{r^2+s^2}$.  Cauchy--Schwarz and AM--GM give
\[
 5r-\sqrt3s\geq-2\sqrt7\rho,\qquad a+b\geq2\rho.
\]
Consequently
\[
 \lambda\geq(\rho-(\sqrt7-1))^2+2(\sqrt7-1)>0.
\]
\end{proof}

The remaining algebra is organized by the Gram defect $\delta$.  Define
\begin{equation}\label{eq:Ndef}
 N=\Rea\left(P_0\det\Theta-\ell^*G\adj(\Theta)\ell\right).
\end{equation}
After the substitution $b=(r^2+s^2+\delta)/a$, multiplication by $a^3$
clears the denominators in $\Rea\det\Theta$ and $N$.  Their coefficient
expansions are therefore polynomial identities in
$\mathbb Q(\sqrt3)[a,r,s,\delta]$.  The representation of $N_0$ below
temporarily introduces square roots only to expose its sign.

\begin{lemma}[Quadratic-form and minimum certificates]\label{lem:quadratic-certs}
For $a>0$ and $\delta\geq0$,
\begin{equation}\label{eq:detTheta-expansion}
 a^3\Rea\det\Theta=D_0+\delta D_1+\delta^2D_2,
\end{equation}
where $D_0,D_1,D_2>0$.  More precisely,
$D_0=a^3\lambda_0H_0$, where $\lambda_0$ denotes~\eqref{eq:lambda}
with $b=(r^2+s^2)/a$ and
$H_0=2\lambda_0-6+4(r^2+s^2)>0$.
Furthermore
\begin{equation}\label{eq:N-expansion}
 a^3N=N_0+\delta N_1+\delta^2N_2+\delta^3N_3,
\end{equation}
with $N_0\geq0$ and $N_1,N_2,N_3>0$.
\end{lemma}

\begin{proof}
Substitution of $b=(r^2+s^2+\delta)/a$ in
\eqref{eq:Theta-matrix} and~\eqref{eq:Ndef} gives the two polynomial
expansions.  Their exact coefficients are recorded in
\ref{app:det-numerator}; here we isolate the geometric content of
their signs.

Put
$\rho=(r^2+s^2)^{1/2}$ and $k=\sqrt7-1$.  On the boundary $\delta=0$, the
two estimates in Lemma~\ref{lem:lambda-positive} give
\[
 \lambda_0\geq(\rho-k)^2+2k.
\]
Consequently
\begin{equation}\label{eq:H0-lower}
 H_0\geq6\rho^2-4k\rho+6
 =6\left(\rho-\frac{k}{3}\right)^2+\frac{2+4\sqrt7}{3}>0.
\end{equation}
This proves $D_0>0$.
The positive residual certificates for $D_1$ and $D_2$ are given in
\ref{app:det-numerator}.

It remains conceptually important to identify the constant coefficient in
\eqref{eq:N-expansion}.  Set $\delta=0$.  The first two residual vectors are
then collinear.  Define
\[
 x_0=\sqrt a,\qquad y_0=\frac{g}{\sqrt a},\qquad
 L_0'=x_0\ell_1+y_0\ell_2,
\]
where $\ell_1,\ell_2$ are the components of~\eqref{eq:ell}, evaluated at
$b=(r^2+s^2)/a$.  Then $|x_0|^2=a$, $|y_0|^2=b$, and
$\bar x_0y_0=g$.  For $z\in\CC$, let
\[
 R_{\mathrm{rk1}}(z)=
 \mat{
  a&g&\bar x_0z\\
  \bar g&b&\bar y_0z\\
  x_0\bar z&y_0\bar z&|z|^2
 }.
\]
This is the rank-one Gram matrix generated by $(x_0,y_0,z)$.  Substitution
in~\eqref{eq:P-quadratic} gives the exact scalar quadratic
\begin{equation}\label{eq:rank-one-N0-bridge}
 \mathcal P(R_{\mathrm{rk1}}(z))
 =P_0+2\Rea(\overline{L_0'}z)+H_0|z|^2.
\end{equation}
Indeed, the linear term is
$\ell^*\mat{\bar x_0z&\bar y_0z}^T=\overline{L_0'}z$, and the coefficient
of $|z|^2$ is $H_0=2\lambda_0-6+4(r^2+s^2)$.  By
Lemma~\ref{lem:rank-one-sos}, the left-hand side of
\eqref{eq:rank-one-N0-bridge} is nonnegative for every $z\in\CC$.
Since $H_0>0$, minimizing the right-hand side gives
$|L_0'|^2\leq P_0H_0$.  Substitution in~\eqref{eq:Ndef} now yields the exact
boundary coefficient
\begin{equation}\label{eq:N0-certificate}
 N_0=a^3\lambda_0(P_0H_0-|L_0'|^2)\geq0.
\end{equation}
The exact human-readable certificates for $N_1,N_2,N_3$ occupy the
remaining subsections of~\ref{app:sos}.  They
show respectively that each remaining coefficient is strictly positive,
without numerical approximation.
\end{proof}

\begin{proposition}[Positive-definite residual Gram matrices]
If $R\pd$, then $\mathcal P(R)>0$.
\end{proposition}

\begin{proof}
Equations~\eqref{eq:detTheta-expansion} and~\eqref{eq:Dcoeffs} show
$\Rea\det\Theta>0$.  Direct multiplication gives the exact leading-entry
identity
\begin{equation}\label{eq:GTheta-leading}
 (G\Theta)_{11}=aH_0+\delta(a+1)>0.
\end{equation}
Moreover
\[
 \det(G\Theta)=\det G\det\Theta=\delta\det\Theta.
\]
Because $G\Theta$ is Hermitian, its determinant is real; since $\delta>0$,
this also proves $\det\Theta=\Rea\det\Theta$.  Hence
$\det(G\Theta)=\delta\Rea\det\Theta>0$.  The two leading principal minors
are positive, so the $2\times2$ Hermitian matrix $G\Theta$ is positive
definite.

The critical-point equation for~\eqref{eq:Qzeta} is
$G\Theta\zeta=-G\ell$.  Since $G$ is invertible, it is equivalent to
$\Theta\zeta=-\ell$, and therefore
\[
 \zeta_0=-\frac{\adj(\Theta)\ell}{\Rea\det\Theta},
\]
and completing the Hermitian square gives
\[
 \min_\zeta Q(\zeta)=\frac{N}{\Rea\det\Theta}.
\]
By~\eqref{eq:N-expansion}, $N>0$ when $\delta>0$.  Therefore
$Q(\zeta)>0$ for all $\zeta$.  Lemma~\ref{lem:lambda-positive} and the Schur
slack~\eqref{eq:schur-slack} now give
\[
 \mathcal P(R)=Q(\zeta)+\lambda(c-\zeta^*G\zeta)>0.
\]
\end{proof}

\subsection{Closure of the Gram cone}

\begin{proof}[Proof of Theorem~\ref{thm:Ppositive}]
For $R\psd$ and $t>0$, the matrix $R+tI$ is positive definite, and hence
$\mathcal P(R+tI)>0$ by the preceding proposition.  Since
$\mathcal P$ is a polynomial in the real and imaginary parts of the entries,
letting $t\downarrow0$ gives $\mathcal P(R)\geq0$.
\end{proof}

\begin{proof}[Proof of Theorem~\ref{thm:four-vector}]
If all coordinates of $q$ are nonzero, the residual matrix
\eqref{eq:residual-gram} is positive semidefinite.  Theorem~\ref{thm:Ppositive}
and~\eqref{eq:normalization-identity} give~\eqref{eq:four-vector}.

For arbitrary $q$, choose phases $p_i$ of modulus one, taking
$p_i=q_i/|q_i|$ when $q_i\neq0$ and $p_i=1$ otherwise.  Convexly interpolate
the block matrix~\eqref{eq:completion} with the rank-one correlation
completion generated by $(p_1,p_2,p_3,1)$.  Its new last column is
$q_i(t)=(1-t)q_i+tp_i$, which is nonzero for $0<t<1$.  The interpolated
matrix remains positive semidefinite, so the already proved case applies.
Letting $t\downarrow0$ and using continuity proves the result.
\end{proof}

\subsection{Completion of the main theorem}\label{sec:completion}

We now combine the geometric inequality with the finite character
classification.

\begin{proposition}[The two non-real $A_4$ characters]\label{prop:a4-characters}
For every $4\times4$ positive-semidefinite correlation matrix $A$,
\[
 d_{\chi_\omega}^{A_4}(A)\leq\per A,\qquad
 d_{\bar\chi_\omega}^{A_4}(A)\leq\per A.
\]
\end{proposition}

\begin{proof}
Write $A$ in the block form~\eqref{eq:completion} with
$B=B(a,b,d)$ and $q=(c,e,f)^T$.  Lemma~\ref{lem:a4-expansion} and
Theorem~\ref{thm:four-vector} give the first inequality.  An odd simultaneous
relabeling preserves positivity and the permanent and interchanges the two
non-real characters, giving the second.
\end{proof}

\begin{proof}[Proof of Theorem~\ref{thm:main}]
By Lemma~\ref{lem:correlation-reduction}, it is enough to consider
correlation matrices.  Lemma~\ref{lem:subgroup-classification} and the
transport Lemma~\ref{lem:transport} reduce an arbitrary subgroup and
irreducible representation to one of the listed character cases.
Proposition~\ref{prop:elementary-cases} proves thirty-five of them and
Proposition~\ref{prop:a4-characters} proves the remaining two.
Lemma~\ref{lem:reality} shows that the compared quantities are real.  Finally,
Lemma~\ref{lem:scaling} returns from correlation matrices to the full
positive-semidefinite cone.
\end{proof}

\section{An adjugate-matrix consequence}\label{sec:matrix-consequence}

The four-vector inequality has a useful matrix consequence.  It is recorded
after the proof of Theorem~\ref{thm:main} to make clear that it is a
consequence of the residual-vector argument, not an additional step in it.

\begin{lemma}[Ellipsoid-to-adjugate certificate]\label{lem:ellipsoid}
Let $B,K\in M_m(\CC)$ with $B\psd$ and $K=K^*$, and let $c\in\RR$.  If
\[
 \mat{B&q\\q^*&1}\psd\quad\Longrightarrow\quad c+q^*Kq\geq0
\]
for every $q\in\CC^m$, then
\begin{equation}\label{eq:adj-certificate}
 (\det B)K+c\,\adj(B)\psd.
\end{equation}
\end{lemma}

\begin{proof}
If $\det B=0$, the choice $q=0$ gives $c\geq0$ and
$\adj(B)\psd$, so~\eqref{eq:adj-certificate} follows.  To justify the latter
fact, if $\operatorname{rank}B\leq m-2$, then $\adj(B)=0$; if
$\operatorname{rank}B=m-1$, spectral decomposition shows that $\adj(B)$ is
the product of the positive eigenvalues times the orthogonal projection onto
$\ker B$.  If $\det B>0$, then $B\pd$ and the block condition is
$q^*B^{-1}q\leq1$.  Given $y\neq0$, set
\[
 q=\frac{y}{\sqrt{y^*B^{-1}y}}.
\]
The hypothesis, multiplied by $y^*B^{-1}y$, says
$y^*(K+cB^{-1})y\geq0$.  Hence $K+cB^{-1}\psd$.  Multiplication by
$\det B>0$ and the identity $\adj(B)=(\det B)B^{-1}$ prove the claim.
\end{proof}

Applied to~\eqref{eq:c0}--\eqref{eq:K}, Lemma~\ref{lem:ellipsoid} gives the
positive-semidefinite certificate
\begin{equation}\label{eq:a4-certificate}
 (\det B(u,v,w))K(u,v,w)+c_0(u,v,w)\adj(B(u,v,w))\psd.
\end{equation}

\appendix

\section{Complete case data}\label{app:character-data}

\subsection{Character tables}

This appendix records the non-cyclic character data used in the finite
character reduction.  For $S_4$, use the conjugacy-class ordering
\[
 1,\quad(12),\quad(12)(34),\quad(123),\quad(1234).
\]
Its irreducible character table is
\begin{equation}\label{eq:s4-table}
\begin{array}{c|rrrrr}
 &1&(12)&(12)(34)&(123)&(1234)\\\hline
{[4]}    &1& 1& 1& 1& 1\\
{[1^4]}  &1&-1& 1& 1&-1\\
{[31]}   &3& 1&-1& 0&-1\\
{[211]}  &3&-1&-1& 0& 1\\
{[22]}   &2& 0& 2&-1& 0
\end{array}
\end{equation}

For $A_4$, let $C_+$ and $C_-$ be the two conjugacy classes of
three-cycles.  With $\omega$ as in~\eqref{eq:omega}, the table is
\begingroup
\renewcommand{\arraystretch}{1.15}
\begin{equation}\label{eq:a4-table}
\begin{array}{c|rrrr}
 &1&(12)(34)&C_+&C_-\\\hline
1&1&1&1&1\\
\chi_\omega&1&1&\omega&\omega^2\\
\bar\chi_\omega&1&1&\omega^2&\omega\\
\chi_3&3&-1&0&0
\end{array}
\end{equation}
\endgroup

For $D_8=\langle r,s:r^4=s^2=1,srs=r^{-1}\rangle$, the four linear
characters take independent signs on $r$ and $s$, while its two-dimensional
character is $2,-2,0,0,0$ on
\[
 1,\quad r^2,\quad\{r,r^3\},\quad\{s,r^2s\},\quad\{rs,r^3s\}.
\]
The standard character tables of $C_m$, $C_2^2$, and $S_3$ supply the
remaining cases listed in Section~\ref{sec:elementary-cases}.

\subsection{Normalized functions and case accounting}

For the smaller subgroups, listing the permutation monomials gives
Table~\ref{tab:elementary-character-expansions}.  In the $C_3$ and $S_3$
rows, $z=ad\bar b$.  In the $V_4^{\rm n}$ row,
$r=|a|^2|f|^2$, $s=|b|^2|e|^2$, and $t=|c|^2|d|^2$; the four signs are
the four character sign patterns of $C_2^2$.

\begin{table}[htbp]
\caption{Normalized generalized matrix functions for the smaller subgroup cases.}
\label{tab:elementary-character-expansions}
\centering
\small
\begin{tabularx}{\textwidth}{@{}l>{\raggedright\arraybackslash}X@{}}
\toprule
group & normalized generalized matrix functions\\
\midrule
$1$ & $1$\\
$C_2^{\rm t}$ & $1\pm|a|^2$\\
$C_2^{\rm d}$ & $1\pm|a|^2|f|^2$\\
$C_3$ & $1+2\Rea z$,\quad $1-\Rea z\pm\sqrt3\Ima z$\\
$V_4^{\rm n}$ & $1\pm r\pm s\pm t$ with the four character sign patterns\\
$V_4^{\rm d}$ & $(1\pm|a|^2)(1\pm|f|^2)$\\
$S_3$ & $1+|a|^2+|b|^2+|d|^2+2\Rea z$,\\
& $1-|a|^2-|b|^2-|d|^2+2\Rea z$,\quad $1-\Rea z$\\
\bottomrule
\end{tabularx}
\end{table}

For $C_4$, set $r=|b|^2|e|^2$ and $z=adf\bar c$.  The four normalized
functions are
\begin{equation}\label{eq:c4-rows}
 1+r+2\Rea z,\quad 1-r-2\Ima z,\quad
 1+r-2\Rea z,\quad 1-r+2\Ima z.
\end{equation}
For $D_8$, set $p=|b|^2$ and $q=|e|^2$.  Its five normalized functions are
\begin{align}
 R_+,&\quad
 (1+p)(1+q)-|af+c\bar d|^2,\nonumber\\
 &(1-p)(1-q)-|af-c\bar d|^2,\quad
 R_\wedge,\quad 1-pq.\label{eq:d8-rows}
\end{align}
The five $S_4$ functions are displayed in~\eqref{eq:s4-rows}.  For $A_4$,
the trivial and three-dimensional normalized functions are, respectively,
\begin{equation}\label{eq:a4-real-rows}
 1+D+C,\qquad 1-\frac13D,
\end{equation}
while the two non-real functions are the conjugate pair represented by the
expansion in Lemma~\ref{lem:a4-expansion}.  Thus the numbers of entries are
\[
 1+2+2+3+4+4+4+3+5+4+5=37,
\]
with thirty-five covered by Proposition~\ref{prop:elementary-cases} and the
remaining two by Proposition~\ref{prop:a4-characters}.

\section{Exact algebraic certificates}\label{app:sos}

This appendix records the longer exact expansions and the non-obvious
polynomial certificates used in the geometric proof.  They contain no
numerical approximation.

\subsection{Residual-normalization expansion}
\label{app:normalization-expansion}

With the notation of Section~\ref{sec:normalization} and
$Q=|q_1q_2q_3|^2$, direct substitution in~\eqref{eq:c0} gives
\begin{align}
\frac{c_0(u,v,w)}{Q}={}&
 (1+n_3)|1+h_{12}|^2+(1+n_2)|1+h_{13}|^2
 +(1+n_1)|1+h_{23}|^2\nonumber\\
&+2\Rea\bigl((1-\omega)(1+h_{12})(1+h_{23})
                   (1+\bar h_{13})\bigr).\label{eq:c0-normalized}
\end{align}
For the quadratic part, the diagonal terms and the three upper-triangular
entries of $K$ give
\begin{align}
\frac{q^*Kq}{Q}={}&
 (1+n_2)(1+n_3)+(1+n_1)(1+n_3)+(1+n_1)(1+n_2)\nonumber\\
&+2\Rea\Bigl(
 (1-\omega^2)(1+n_3)(1+h_{12})
 +(1+h_{13})(1+\bar h_{23})\nonumber\\
&\qquad +(1+h_{12})(1+h_{23})
 +(1-\omega)(1+n_2)(1+h_{13})\nonumber\\
&\qquad +(1-\omega^2)(1+n_1)(1+h_{23})
 +(1+h_{13})(1+\bar h_{12})\Bigr).
\label{eq:K-normalized}
\end{align}
Collecting these two formulas by degree gives
\eqref{eq:normalization-identity}.

\subsection{Determinant and numerator coefficients}\label{app:det-numerator}

After $b=(r^2+s^2+\delta)/a$, direct coefficient collection in
\eqref{eq:Theta-matrix} gives~\eqref{eq:detTheta-expansion} with
\begin{align}
 D_0&=a^3\lambda_0H_0,\nonumber\\
 D_1&=a\bigl((4+11a)(r^2+s^2)+a(20+a)r\nonumber\\
 &\qquad\quad-\sqrt3a(4+3a)s
          +a(18+10a+3a^2)\bigr),\label{eq:Dcoeffs}\\
 D_2&=a(a+1)(a+2).
\end{align}
The sign of $D_0$ was proved in the main text.  For $D_1$, set
\[
 Q=4+11a,\quad L_r=a(20+a),\quad
 L_s=-\sqrt3a(4+3a),\quad C=a(18+10a+3a^2).
\]
The exact identity
\begin{equation}\label{eq:two-variable-completion}
\begin{aligned}
 4Q\bigl(Q(r^2+s^2)+L_rr+L_ss+C\bigr)
  &=(2Qr+L_r)^2+(2Qs+L_s)^2\\
  &\quad+4QC-L_r^2-L_s^2
\end{aligned}
\end{equation}
reduces positivity to
\[
 4QC-L_r^2-L_s^2
 =4a(72+126a+94a^2+26a^3)>0.
\]
Thus $D_1>0$, and $D_2>0$ is immediate.

Coefficient collection in~\eqref{eq:Ndef} gives~\eqref{eq:N-expansion}.
The constant coefficient is~\eqref{eq:N0-certificate}.  At the other end,
\[
 N_3=(a+2)(a^2+6)>0.
\]
For $N_2$, one obtains
\begin{align*}
 N_2={}&36(r^2+s^2)
 +a(120+84r-36\sqrt3s+36(r^2+s^2))\\
 &+a^2(60+24r-24\sqrt3s+16(r^2+s^2))\\
 &+a^3(16+2(r^2+s^2))+2a^4.
\end{align*}
Thus $N_2=Q_2(r^2+s^2)+L_{2r}r+L_{2s}s+C_2$, where
\begin{align*}
 Q_2&=36+36a+16a^2+2a^3,\\
 L_{2r}&=84a+24a^2,\qquad
 L_{2s}=-\sqrt3(36a+24a^2),\\
 C_2&=120a+60a^2+16a^3+2a^4.
\end{align*}
Applying~\eqref{eq:two-variable-completion} gives
\[
 4Q_2C_2-L_{2r}^2-L_{2s}^2
 =16a(a^6+16a^5+112a^4+318a^3+588a^2+936a+1080)>0,
\]
and hence $N_2>0$.

For $N_1$, put $X=r$, $Y=\sqrt3s$ and
$\nu=X^2+Y^2/3$.  Then
\begin{equation}\label{eq:N1-B}
 N_1=36\nu^2+aB_1+a^2B_2+a^3B_3+a^4B_4+6a^5,
\end{equation}
where
\begin{align*}
 B_1&=\nu(240+168X-72Y+54\nu),\\
 B_2&=432+624X-240Y+492X^2+108Y^2-168XY\\
    &\qquad +(156X-84Y)\nu+38\nu^2,\\
 B_3&=168+72X-168Y+54X^2+46Y^2-84XY\\
    &\qquad +(20X-24Y)\nu+6\nu^2,\\
 B_4&=60+24X-24Y+14\nu.
\end{align*}
The next two subsections prove the required signs.

\subsection{Direct square completions}\label{app:direct-completions}

The coefficient $B_2$ has the exact decomposition
\begin{align*}
 B_2={}&38\left(\nu+\frac{39}{19}X-\frac{21}{19}Y+\frac32\right)^2
 +\frac{448}{19}\left(Y+\frac3{32}X-\frac{1083}{448}\right)^2\\
 &+\frac{3483}{16}\left(X+\frac{2137}{2322}\right)^2
 +\frac{131797}{5418}>0.
\end{align*}
Likewise
\begin{align*}
 B_3+252={}&6\left(\nu+\frac53X-2Y-\frac12\right)^2
 +24\left(Y-\frac{11}{12}X-\frac{15}{4}\right)^2\\
 &+\frac{139}{6}\left(X-\frac{249}{139}\right)^2
 +\frac{1851}{278}>0.
\end{align*}
Further elementary completions give
\[
 B_1\geq\frac{112}{3}\nu\geq0,\qquad
 B_4\geq\frac{132}{7}>0.
\]

\subsection{The conditional middle-coefficient certificate}
\label{app:conditional-certificate}

Set
\[
 \Sigma=4B_2B_4+252B_3.
\]
An exact Gram decomposition is
\begin{equation}\label{eq:sigma-gram}
 \Sigma=2128\nu\left(\nu+\frac{387}{133}X-
 \frac{261}{133}Y+\frac{77}{20}\right)^2
 +\frac1{59850}m^TQm,
\end{equation}
where $m=(1,Y,Y^2,X,XY,X^2)^T$ and
\begingroup
\scriptsize
\[
Q=\mat{
8739057600&-4231634400&478800000&6265576800&-1795500000&47880000\\
-4231634400&2519799114&-385003080&-2526627600&807975000&239400000\\
478800000&-385003080&79782560&281630160&-81252000&-35910000\\
6265576800&-2526627600&281630160&10931106942&-2359670040&2303554680\\
-1795500000&807975000&-81252000&-2359670040&685371360&-243756000\\
47880000&239400000&-35910000&2303554680&-243756000&1122611040
}.
\]
\endgroup
Exact rational $LDL^T$ elimination has positive pivots
\begin{align*}
&8739057600,\quad \frac{79557301866}{169},\quad
\frac{1112696849800960}{299087601},\\
&\frac{2624099206828303859949}{869294413907},\quad
\frac{2531308970961028469920080}{115380521779374043},\\
&\frac{1406694633431943161674826341030}
{10547120712337618624667}.
\end{align*}
Thus $Q\pd$ and $m^TQm>0$ because the first coordinate of $m$ is one.
Equation~\eqref{eq:sigma-gram} proves $\Sigma>0$.

If $B_3\geq0$, then $B_2+aB_3+a^2B_4>0$ is immediate.  If $B_3<0$, the
identity
\[
 4B_2B_4-B_3^2=(4B_2B_4+252B_3)+(-B_3)(B_3+252)
\]
shows that the discriminant of the quadratic
$B_2+aB_3+a^2B_4$ is negative.  Hence that quadratic is positive for every
$a>0$.  This completes the proof that $N_1>0$.

\section{Formal verification}\label{app:formal-verification}

The theorem has been formalized in Lean 4.19.0 with Mathlib 4.19.0.  The
formal development covers the subgroup and character enumeration, the cycle
expansions, the residual normalization identity, the rational Gram
certificate, and the passage from correlation matrices to the full
positive-semidefinite cone.  The public theorem is
\texttt{normalized\_permanental\_dominance\_n4} in the module
\texttt{PermanentalDominance.N4.}\allowbreak
\texttt{NormalizedPermanentalDominance}.
A detailed module map and an axiom audit are included with the source.  In
the module \path{A4GeometricHessian}, the machine-checked counterparts of
the three core coefficient identities are as follows:
\begin{itemize}
\item \path{hessianDeterminant_expansion} for
      \eqref{eq:detTheta-expansion};
\item \path{minimumNumerator_expansion} for \eqref{eq:N-expansion};
\item \path{collinearNumerator_scaled_identity} for
      \eqref{eq:N0-certificate}.
\end{itemize}
The matrix denoted by $\Theta$ here is named \texttt{T} in that module.  The
rank-one identity and the exact interior minimum are theorems
\path{geometricP_collinear_identity} and \path{minimumPoint_value}, located
in \path{A4GeometricRankOne} and \path{A4GeometricMinimum}, respectively.
These references locate the machine-checked counterparts; the identities
and sign certificates used by the mathematical proof remain displayed
above.

The formal source used here was extracted from the immutable merge commit
\href{https://github.com/EndlessLethe/math/commit/c5e4d8c61f08d05ffe7c28bfa1ea732d8827b485}
{\texttt{c5e4d8c61f08}} and is mirrored in the companion repository.

\section*{Acknowledgments}
During the preparation of this work, the author used OpenAI ChatGPT and Codex
to assist with mathematical exploration, development of the Lean
formalization, literature organization, manuscript drafting, language
editing, and \LaTeX\ preparation.  The author reviewed and edited all
generated content and takes full responsibility for the accuracy and
integrity of the work.

\section*{Data and code availability}
No external data were used in this study.  The accompanying Lean
formalization is available from the
\href{https://github.com/EndlessLethe/PermanentalDominanceConjecture}
{companion repository}, together with the manuscript source.  The immutable
formalization provenance is the commit identified in the preceding section.

\section*{Funding and competing interests}
This research did not receive any specific grant from funding agencies in the
public, commercial, or not-for-profit sectors.  The author declares that
there are no known competing financial interests or personal relationships
that could have appeared to influence the work reported in this paper.

\bibliography{references}

\begin{thebibliography}{10}

\bibitem{barretthallloewy1999}
{\sc W.~Barrett, H.~T. Hall, and R.~Loewy}, {\em The cone of class function
  inequalities for the {4}-by-{4} positive semidefinite matrices}, Proceedings
  of the London Mathematical Society, 79 (1999), pp.~107--130,
  \url{https://doi.org/10.1112/S0024611599011909}.

\bibitem{cheonwanless2005}
{\sc G.-S. Cheon and I.~M. Wanless}, {\em An update on {Minc}'s survey of open
  problems involving permanents}, Linear Algebra and its Applications, 403
  (2005), pp.~314--342, \url{https://doi.org/10.1016/j.laa.2005.02.030}.

\bibitem{gellerknill2026}
{\sc S.~Geller and E.~Knill}, {\em Measuring multiparticle indistinguishability
  with the generalized bunching probability}, Physical Review A, 113 (2026),
  p.~042606, \url{https://doi.org/10.1103/1x5l-jd26}.

\bibitem{heyfron1988}
{\sc P.~Heyfron}, {\em Immanant dominance orderings for hook partitions},
  Linear and Multilinear Algebra, 24 (1988), pp.~65--78.

\bibitem{james1992}
{\sc G.~D. James}, {\em Immanants}, Linear and Multilinear Algebra, 32 (1992),
  pp.~197--210, \url{https://doi.org/10.1080/03081089208818163}.

\bibitem{jingliu2025}
{\sc N.~Jing, Y.~Liu, and J.~Zhang}, {\em Immanant inequalities and weight
  spaces}, 2025, \url{https://doi.org/10.48550/arXiv.2508.20382},
  \url{https://arxiv.org/abs/2508.20382}.

\bibitem{lieb1966}
{\sc E.~H. Lieb}, {\em Proofs of some conjectures on permanents}, Journal of
  Mathematics and Mechanics, 16 (1966), pp.~127--134,
  \url{https://doi.org/10.1512/iumj.1967.16.16008}.

\bibitem{marcus1963}
{\sc M.~Marcus}, {\em The permanent analogue of the {Hadamard} determinant
  theorem}, Bulletin of the American Mathematical Society, 69 (1963),
  pp.~494--496, \url{https://doi.org/10.1090/S0002-9904-1963-10975-1}.

\bibitem{merris1987}
{\sc R.~Merris}, {\em The permanental dominance conjecture}, in Current Trends
  in Matrix Theory, R.~Grone and F.~Uhlig, eds., North-Holland, New York, 1987,
  pp.~213--223.

\bibitem{merriswatkins1985}
{\sc R.~Merris and W.~Watkins}, {\em Inequalities and identities for
  generalized matrix functions}, Linear Algebra and its Applications, 64
  (1985), pp.~223--242, \url{https://doi.org/10.1016/0024-3795(85)90279-4}.

\bibitem{pate1991}
{\sc T.~H. Pate}, {\em Partitions, irreducible characters, and inequalities for
  generalized matrix functions}, Transactions of the American Mathematical
  Society, 325 (1991), pp.~875--894,
  \url{https://doi.org/10.1090/S0002-9947-1991-0998356-3}.

\bibitem{pate1998}
{\sc T.~H. Pate}, {\em Row appending maps, {$\Psi$}-functions, and immanant
  inequalities for hermitian positive semi-definite matrices}, Proceedings of
  the London Mathematical Society, 76 (1998), pp.~307--358,
  \url{https://doi.org/10.1112/S0024611598000100}.

\bibitem{pate1999}
{\sc T.~H. Pate}, {\em Tensor inequalities, {$\xi$}-functions and inequalities
  involving immanants}, Linear Algebra and its Applications, 295 (1999),
  pp.~31--59, \url{https://doi.org/10.1016/S0024-3795(99)00035-X}.

\bibitem{rodtes2024}
{\sc K.~Rodtes}, {\em Some remarks on permanental dominance conjecture},
  Advances in Applied Mathematics, 160 (2024), p.~102758,
  \url{https://doi.org/10.1016/j.aam.2024.102758}.

\bibitem{schur1918}
{\sc I.~Schur}, {\em {\"U}ber endliche gruppen und hermitesche formen},
  Mathematische Zeitschrift, 1 (1918), pp.~184--207,
  \url{https://doi.org/10.1007/BF01465095}.

\bibitem{soules1994}
{\sc G.~W. Soules}, {\em An approach to the permanental-dominance conjecture},
  Linear Algebra and its Applications, 201 (1994), pp.~211--229,
  \url{https://doi.org/10.1016/0024-3795(94)90117-1}.

\bibitem{tabata2015}
{\sc R.~Tabata}, {\em The values of the generalized matrix functions of {$3
  \times 3$} matrices}, Hiroshima Mathematical Journal, 45 (2015), pp.~1--8,
  \url{https://doi.org/10.32917/hmj/1428365051}.

\bibitem{wanless2022}
{\sc I.~M. Wanless}, {\em {Lieb}'s permanental dominance conjecture}, in The
  Physics and Mathematics of Elliott Lieb: The 90th Anniversary, vol.~2, EMS
  Press, Berlin, 2022, pp.~501--516, \url{https://doi.org/10.4171/90-2/48}.

\bibitem{zhang2016}
{\sc F.~Zhang}, {\em An update on a few permanent conjectures}, Special
  Matrices, 4 (2016), pp.~305--316,
  \url{https://doi.org/10.1515/spma-2016-0030}.

\end{thebibliography}

\end{document}